\documentclass[11pt]{amsart}
\usepackage[utf8]{inputenc}

\usepackage{amssymb,amsthm,amsmath,mathrsfs}
\usepackage{enumerate}
\usepackage{graphicx,color}
\usepackage{setspace}
\usepackage[hidelinks]{hyperref}
\usepackage{setspace}

\usepackage[paper=a4paper, left=1.2in, right=1.2in, top=1.2in, bottom=1.2in]{geometry}
\usepackage{tikz-cd}
\usepackage{tikz}
\usetikzlibrary{calc}
\usetikzlibrary{decorations.markings}

\newcommand{\dd}{\mathrm{d}}
\newcommand{\E}{\mathbb{E}}

\newcommand{\R}{\mathbb{R}}
 
\newcommand{\Z}{\mathbb{Z}}
\newcommand{\e}{\varepsilon}

\newcommand{\p}[1]{\mathbb{P}\left( #1 \right)}
\newcommand{\scal}[2]{\!\left\langle #1, #2 \right\rangle\!}
\newcommand{\red}{}

\newcommand{\msf}{\mathsf}

\newtheorem{theorem}{Theorem}
\newtheorem{lemma}[theorem]{Lemma}

\theoremstyle{remark}
\newtheorem{remark}{Remark}

\theoremstyle{definition}

\title{
Distributional stability of the Szarek\\ and Ball inequalities
}

\author{Alexandros Eskenazis}
\address{(A.~E.) CNRS, Institut de Math\'ematiques de Jussieu, Sorbonne Universit\'e, France and Trinity College, University of Cambridge, UK.}
\email{alexandros.eskenazis@imj-prg.fr, ae466@cam.ac.uk}

\author{Piotr Nayar}
\address{(P.~N.) University of Warsaw, 02-097 Warsaw, Poland.}
\email{nayar@mimuw@edu.pl} 

\author{Tomasz Tkocz}
\address{(T.~T.) Carnegie Mellon University, Pittsburgh, PA 15213, USA.}
\email{ttkocz@andrew.cmu.edu} 

\thanks{This material is based upon work supported by the NSF grant DMS-1929284 while A.~E.~was in residence at ICERM for the Harmonic Analysis and Convexity program. P.N.’s research was supported by the National Science Centre, Poland, grant 2018/31/D/ST1/0135. T.T.’s research was supported by the NSF grant DMS-1955175.}

\begin{document}

\begin{abstract}
We prove an extension of Szarek's optimal Khinchin inequality (1976) for distributions close to the Rademacher one, when all the weights are uniformly bounded by a $1/\sqrt2$ fraction of their total $\ell_2$-mass. We also show a similar extension of the probabilistic formulation of Ball's cube slicing inequality (1986). These results establish the distributional stability of these optimal Khinchin-type inequalities. The underpinning to such estimates  is the Fourier-analytic approach going back to Haagerup (1981).  
\end{abstract}

\maketitle


\begin{footnotesize}
\noindent {\em 2010 Mathematics Subject Classification.} Primary 60E15; Secondary 42A38, 26D15, 60G50.


\noindent {\em Key words. Khinchin inequality, sums of independent random variables, moment comparison, cube slicing.} 
\end{footnotesize}


\section{Introduction} 

Let $\e_1, \e_2, \dots$ be independent identically distributed (i.i.d.) Rademacher random variables, that is, symmetric random signs satisfying $\p{\e_j = \pm 1} = \frac{1}{2}$. Motivated by his study of bilinear forms on infinitely many variables, Littlewood conjectured in \cite{Lit} (see also \cite{Hall}) the following inequality: for every $n \geq 1$ and every unit vector $a$ in $\R^n$, we have
\begin{equation}\label{eq:Sza}
\E\left|\sum_{j=1}^n a_j\e_j\right| \geq \E\left|\frac{\e_1+\e_2}{\sqrt{2}}\right| = \frac{1}{\sqrt{2}},
\end{equation}
which is clearly best possible.
Not until 46 years after it had been posed, was this proved by Szarek in \cite{Sza}.  His result was later generalised in a stunning way to the setting of vector-valued coefficients $a_j$ in arbitrary normed space by Lata\l a and Oleszkiewicz in \cite{LO-best} (see also \cite[Section~4.2]{NT} for a modern presentation of their proof using discrete Fourier analysis). Szarek's original proof was based mainly on an intricate inductive scheme (see also \cite{Tom}). Note that \eqref{eq:Sza} holds trivially if $\|a\|_\infty = \max_j |a_j| \geq \frac{1}{\sqrt2}$, for if, say we have $|a_1| \geq \frac{1}{\sqrt2}$, then thanks to independence and convexity,
\[
\E\left|\sum_{j=1}^n a_j\e_j\right| \geq \E\left|a_1\e_1 + \E\sum_{j=2}^n a_j\e_j\right| = \E|a_1\e_1| = |a_1| \geq \frac{1}{\sqrt{2}}.
\]
Haagerup in his pioneering work \cite{Haa} on Khinchin inequalities offered a very different approach to the nontrivial regime $\|a\|_\infty \leq \frac{1}{\sqrt2}$, using classical Fourier-analytic integral representations along with tricky estimates for a special function.

Taking that route, the point of this paper is to illustrate the robustness of Haagerup's method and extend \eqref{eq:Sza} to i.i.d.~sequences of random variables whose distribution is \emph{close} to the Rademacher one in the $\msf{W}_2$-Wasserstein distance. Using the same framework, we also treat Ball's cube slicing inequality from \cite{Ba} which asserts that the maximal-volume hyperplane section of the cube $[-1,1]^n$ in $\R^n$ is attained at $(1,1,0,\dots, 0)^\perp$. This can be equivalently stated in probabilistic terms as an inequality akin to \eqref{eq:Sza} as follows (see, e.g.  equation (2) in \cite{CKT}). Let $\xi_1, \xi_2, \dots$ be i.i.d.~random vectors uniform on the unit Euclidean sphere in $\R^3$. For every $n \geq 1$ and every unit vector $a$ in $\R^n$, we have
\begin{equation}\label{eq:Ball}
\E\left[\left|\sum_{j=1}^n a_j\xi_j\right|^{-1}\right] \leq \E\left[\left|\frac{\xi_1+\xi_2}{\sqrt{2}}\right|^{-1}\right] = \sqrt{2},
\end{equation}
where here and throughout $|\cdot|$ denotes the standard Euclidean norm.

Szarek's inequality \eqref{eq:Sza}, {\red Ball's} inequality \eqref{eq:Ball}, as well as these extensions fall under the umbrella of so-called Khinchin-type inequalities. The archetype was Khinchin's result asserting that all $L_p$ norms of Rademacher sums $\sum a_j\e_j$ are comparable to its $L_2$-norm, established in his work \cite{Khi} on the law of the iterated logarithm (and perhaps discovered independently by Littlewood in \cite{Lit}). Due to the intricacies of the methods involved, sharp Khinchin inequalities are known only for a handful of distributions, most notably random signs (\cite{Haa, NO}), but also uniforms (\cite{BC, CGT, CKT, CST,  Ko, KoKw, LO}), type L (\cite{HNT, New}), Gaussian mixtures (\cite{AH, ENT1}), marginals of $\ell_p$-balls (\cite{BN, ENT2}), or distributions with good spectral properties (\cite{KLO, O}). The present work makes a first step towards more general distributions satisfying only a {closeness}-type assumption instead of imposing structural properties. Viewing sharp Khinchin-type inequalities as maximization problems for functionals on the sphere,  our results assert, perhaps surprisingly, the fact that such inequalities are stable with respect to perturbations of the law of the underlying random vectors.  These \emph{distributional stability} results are novel in the context of optimal probabilistic inequalities.


\section{Main results}

For $p > 0$ and a random vector $X$ in $\R^d$, we denote its $L_p$-norm with respect to the standard Euclidean norm $|\cdot|$ on $\R^d$  by $\|X\|_p = (\E|X|^p)^{1/p}$, whereas for a (deterministic) vector $a$ in $\R^n$, $\|a\|_\infty = \max_{j \leq n} |a_j|$ is its $\ell_\infty$-norm. We say that the random vector $X$ in $\R^d$ is symmetric if $-X$ has the same distribution as $X$. We also recall that the vector $X$ is called rotationally invariant if for every orthogonal map $U$ on $\R^d$, $UX$ has the same distribution as $X$. Equivalently, $X$ has the same distribution as $|X|\xi$, where $\xi$ is uniformly distributed on the unit sphere $\mathbb{S}^{d-1}$ in $\R^d$ and independent of $|X|$. Recall that the $\msf{W}_2$-Wasserstein distance $\msf{W}_2(X,Y)$ between (the distributions of) two random vectors $X$ and $Y$ in $\R^d$ is defined as $\inf_{(X',Y')} \|X'-Y'\|_2$, where the infimum is taken over all couplings of $X$ and $Y$, that is, all random vectors $(X',Y')$ in $\mathbb{R}^{2d}$ such that $X'$ has the same distribution as $X$ and $Y'$ has the same distribution as $Y$.  

Our first result is an extension of Szarek's inequality \eqref{eq:Sza} which reads as follows.

\begin{theorem}\label{thm:mainS}
There is a positive universal constant $\delta_0$ such that if we let $X_1, X_2, \dots$ be i.i.d. ~symmetric random variables satisfying 
\begin{equation}\label{eq:assump-S}
\big\||X_1|-1\big\|_2 \leq \delta_0,
\end{equation}
then for every $n \geq 3$ and unit vectors $a$ in $\R^n$ with $\|a\|_\infty \leq \frac{1}{\sqrt{2}}$, we have
\begin{equation}\label{eq:mainS}
\E\left|\sum_{j=1}^n a_jX_j\right| \geq \E\left|\frac{X_1+X_2}{\sqrt{2}}\right|.
\end{equation}
Moreover, we can take $\delta_0 = 10^{-4}$.
\end{theorem}

Note that left hand side of \eqref{eq:assump-S}  is nothing but the $\msf{W}_2$-Wasserstein distance between the distribution of $X_1$  and the Rademacher distribution since $|x\pm 1| \geq \big||x|-1\big|$ for $x \in \R$ and thus the optimal coupling of the two distributions is $\big(X_1,\mathrm{sign}(X_1)\big)$.

Our second main result provides an analogous extension for Ball's inequality \eqref{eq:Ball}.

\begin{theorem}\label{thm:mainB}
Let $X_1, X_2, \dots$ be i.i.d. ~symmetric random vectors in $\R^3$. Suppose their common characteristic function $\phi(t) = \E e^{i\scal{t}{X_1}}$ satisfies
\begin{equation}\label{eq:assump-decay}
\left|\phi(t)\right| \leq \frac{C_0}{|t|}, \qquad t \in \R^3\setminus\{0\},
\end{equation}
for some constant $C_0 > 0$. Assume that
\begin{equation}\label{eq:assump-B}
\msf{W}_2(X_1, \xi) \leq 10^{-38}C_1^{-9} \min\big\{(\E|X_1|^3)^{-6},1\big\},
\end{equation}
where $C_1 = \max\{C_0,1\}$ and $\xi$ is a random vector uniform on the unit Euclidean sphere $\mathbb{S}^2$ in $\R^3$.
Then for every $n \geq 3$ and unit vectors $a$ in $\R^n$ with $\|a\|_\infty \leq \frac{1}{\sqrt{2}}$, we have
\begin{equation}\label{eq:mainB}
\E\left|\sum_{j=1}^n a_jX_j\right|^{-1} \leq \E\left|\frac{X_1+X_2}{\sqrt{2}}\right|^{-1}.
\end{equation}
\end{theorem}

Plainly, if we know that $X_1$ and $\xi$ are sufficently close in $\msf{W}_3$, then the parameter $\E|X_1|^{3}$ in \eqref{eq:assump-B} is redundant. In contrast to Theorem \ref{thm:mainS}, here the {closeness} assumption \eqref{eq:assump-B} is put in terms of two parameters of the distribution: its third moment and the polynomial decay of its characteristic function. It is not clear whether this is essential. At the technical level of our proofs, the third moment is needed to carry out a certain Gaussian approximation, whilst the decay assumption has to do with an a priori lack of integrability in the Fourier-analytic representation of the $L_{-1}$ norm (as opposed to the $L_1$-norm handled in Theorem~\ref{thm:mainS}). 

On the other hand, neither of these is very restrictive. In particular, if $X_1$ has a density $f$ on $\R^3$ vanishing at $\infty$ whose gradient is integrable, then
\begin{align*}
|t||\phi(t)| \leq \sum_{j=1}^3 |t_j\phi(t)| = \sum_{j=1}^3 \left|\int_{\R^3} t_je^{i\scal{t}{x}}f(x)\dd x\right| &= \sum_{j=1}^3 \left|\int_{\R^3} ie^{i\scal{t}{x}}\partial_j f(x)\dd x\right| \\
&\leq \sqrt{3}\int_{\R^3}|\nabla f(x)| \dd x,
\end{align*}
so \eqref{eq:assump-decay} holds with $C_0 = \sqrt{3}\int_{\R^3}|\nabla f|$.

Another natural sufficient condition is the rotational invariance of $X_1$: if, say, $X_1$ has the same distribution as $R\xi$, for a nonnegative random variable $R$ and an independent of it random vector $\xi$ uniform on the unit sphere $\mathbb{S}^2$, then Archimedes' Hat-Box theorem implies that $\scal{t}{R\xi}$, conditioned on the value of $R$, is uniform on $[-R|t|,R|t|]$ and thus
\[
|\phi(t)| = |\E_R\E_\xi e^{i\scal{t}{R\xi}}| = \left|\E_R \frac{\sin(R|t|)}{R|t|}\right| \leq \frac{\E R^{-1}}{|t|} = \frac{\E |X_1|^{-1}}{|t|}.
\]
Moreover, in this case $\msf{W}_2(X_1,\xi) = \|R-1\|_2$ (since for every unit vectors $\theta, \theta'$ in $\R^d$ and $R \geq 0$, we have $|R\theta-\theta'| \geq |R-1|$, as is easily seen by squaring). Probabilistically, this is an important special case as it yields results for symmetric unimodal distributions on $\R$. Indeed,  if $X$ is of the form $R\xi$ as above,  for $q > -1$,  we have the identity
\begin{equation}
\E\left|\sum_{j=1}^n a_jX_j\right|^{q} = \E\left|\sum_{j=1}^n a_jR_j\xi_j\right|^{q} = (1+q)\E\left|\sum_{j=1}^n a_jR_jU_j\right|^{q},
\end{equation}
where the $R_j$ are i.i.d.~copies of $R$ and the $U_j$ are i.i.d.~uniform random variables on $[-1,1]$, independent of the $R_j$ (see Proposition 4 in \cite{KK}). The $R_jU_j$ showing up in this formula can have any symmetric unimodal distribution, uniquely defined by the distribution of $R_j$. Thus, if $V_1, V_2, \dots$ be i.i.d. ~symmetric unimodal random variables, Theorem \ref{thm:mainB} then immediately yields a sharp upper bound on $\lim_{q\downarrow-1}(1+q)\E\left|\sum_{j=1}^n a_jX_j\right|^{q}$ for all unit vectors $a$ with $\|a\|_\infty \leq \frac{1}{\sqrt{2}}$ (cf. \cite{CKT, CGT, ENT2, LO}).


A result in the same vein as Theorem \ref{thm:mainB} is K\"onig and Koldobsky's extension \cite{KK} of Ball's cube slicing inequality to product measures with densities satisfying certain regularity and moment assumptions.  Their result also applies specifically to vectors of weights satisfying the small coefficient condition $\|a\|_\infty\leq\tfrac{1}{\sqrt{2}}$.

Approached differently, \emph{full} extensions of \eqref{eq:Sza} and \eqref{eq:Ball} (i.e.~without the small coefficient restriction on $a$) have been obtained in our recent work \cite{ENT} for a very special family of distributions corresponding geometrically to extremal sections and projections of $\ell_p$-balls.

{\red
\subsection*{Acknowledgements.} 
We should very much like to thank an anonymous referee for their careful reading of the manuscript and helpful suggestions, particularly the one leading to Remark \ref{rem:other-q}.
}


\section{Proof of Theorem \ref{thm:mainS}}

Our approach builds on Haagerup's slick Fourier-analytic proof from \cite{Haa}. We let
\begin{equation}
\phi(t) = \E e^{itX_1}, \qquad t \in \R,
\end{equation}
be the characteristic function of $X_1$. Using the elementary Fourier-integral representation
\[
|x| = \frac{1}{\pi}\int_\R (1-\cos(tx))t^{-2}\dd t, \qquad x \in \R,
\]
as well as the symmetry and independence of the $X_j$, we have,
\begin{align} \label{eq:fourier-rep}
\E\left|\sum_{j=1}^n a_jX_j\right| &= \frac{1}{\pi}\int_{\R} \left(1-\text{Re }\E e^{it\sum a_jX_j}\right) t^{-2}\dd t 
= \frac{1}{\pi}\int_{\R} \left(1-\prod_{j=1}^n \phi(a_jt)\right) t^{-2}\dd t
\end{align}
(see also Lemma 1.2 in \cite{Haa}). If $a$ is a unit vector in $\R^n$ with nonzero components, using the AM-GM inequality, we obtain Haagerup's lower bound
\begin{equation}
\E\left|\sum_{j=1}^n a_jX_j\right| \geq \sum_{j=1}^n a_j^2\Psi(a_j^{-2}),
\end{equation}
where
\begin{equation}\label{eq:defPsi}
\Psi(s) = \frac{1}{\pi}\int_\R\left(1 - \left|\phi\left(\frac{t}{\sqrt{s}}\right)\right|^s\right)t^{-2} \dd t, \qquad s > 0.
\end{equation}
(see Lemma 1.3 in \cite{Haa}). The crucial lemma reads as follows.

\begin{lemma}\label{lm:Psi2}
Under the assumptions of Theorem \ref{thm:mainS}, we have $\Psi(s) \geq \Psi(2)$ for every $s\geq2$.
\end{lemma}

If we take the lemma for granted, the proof of Theorem \ref{thm:mainS} is finished because the small coefficient assumption $\|a\|_\infty \leq \frac{1}{\sqrt{2}}$ gives $\Psi(a_j^{-2}) \geq \Psi(2)$ for each $j$, and as a result we get
\[
\E\left|\sum_{j=1}^n a_jX_j\right| \geq \Psi(2) =  \frac{1}{\pi}\int_\R\left(1 - \left|\phi\left(\frac{t}{\sqrt{2}}\right)\right|^2\right)t^{-2} \dd t = \E\left|\frac{X_1+X_2}{\sqrt{2}}\right|,
\]
where the last equality is justified by \eqref{eq:fourier-rep}.

It remains to prove Lemma \ref{lm:Psi2}. To this end, we recall that if the $X_j$ were Rademacher random variables, then the special function $\Psi$ becomes
\begin{equation}\label{eq:defPsi0}
\Psi_0(s) = \frac{1}{\pi}\int_\R\left(1 - \left|\cos\left(\frac{t}{\sqrt{s}}\right)\right|^s\right)t^{-2} \dd t, \qquad s > 0.
\end{equation}
Haagerup showed that for every $s>0$, 
\begin{equation} \label{eq:haa-formula}
\Psi_0(s) = \frac{2}{\sqrt{\pi s}}\frac{\Gamma\left(\frac{s+1}{2}\right)}{\Gamma\left(\frac{s}{2}\right)} = \sqrt{\frac{2}{\pi}}\prod_{k=0}^\infty \Big(1 - 1/(s+2k+1)^2\Big)^{1/2}
\end{equation}
and concluded by the product representation that $\Psi_0$ is strictly increasing.  In particular, Lemma \ref{lm:Psi2} holds in the Rademacher case due to monotonicity. The rest of the proof builds exactly on this observation: we show that the closeness of distributions guarantees that $\Psi$ and $\Psi_0$ are close for, say $s \geq 3$, and that their derivatives are close for $2 \leq s \leq 3$. Crucially, not only do we know that $\Psi_0$ is strictly monotone, but also we can get a good bound on its derivative near the endpoint $s=2$, which we record now for future use.

\begin{lemma}\label{lm:Psi1-bounds}
We have
\begin{equation}
\inf_{2 \leq s \leq 3} \Psi_0'(s) \geq \frac{\zeta(3)-1}{8\sqrt{2}} = 0.01785...
\end{equation}
\end{lemma}
\begin{proof}
Differentiating Haagerup's product expression \eqref{eq:haa-formula} term-by-term yields
\begin{align*}
\Psi_0'(s) & = \frac{\dd}{\dd s}\sqrt{\frac{2}{\pi}}\prod_{k=0}^\infty \Big(1-(s+2k+1)^{-2}\Big)^{1/2} \\ &
= \Psi_0(s)\sum_{k=0}^\infty \Big(1-(s+2k+1)^{-2}\Big)^{-1}(s+2k+1)^{-3} \\ & \geq \Psi_0(2) \sum_{k=0}^\infty (2k+4)^{-3}  = \frac{1}{\sqrt{2}}\frac{\zeta(3)-1}{8}.\qquad\qquad\qedhere
\end{align*}
\end{proof}

\noindent The rest of this section is devoted to the proof of Lemma \ref{lm:Psi2}. \mbox{We break it into several parts.}


\subsection{A uniform bound on the characteristic function}

\begin{lemma}\label{lm:phi-unif}
Let $X$ be a symmetric random variable satisfying \eqref{eq:assump-S}. Then its characteristic function $\phi(t) = \E e^{itX}$ satisfies,
\begin{equation}
\left|\phi(t) - \cos t\right| \leq \frac{\delta_0(\delta_0+2)}{2}t^2, \qquad t \in \R.
\end{equation}
\end{lemma}
\begin{proof}
By symmetry, the triangle inequality and the bound $|\sin u| \leq |u|$, we get
\begin{align*}
\left|\phi(t) - \cos t\right| 
&= \left|\E\left[\cos(t |X|)-\cos t\right]\right| = 2\left|\E\left[\sin\left(t\frac{|X|-1}{2}\right)\sin\left(t\frac{|X|+1}{2}\right)\right]\right| \\
&\leq \frac{t^2}{2}\E\left[\big||X|-1\big|\cdot\big||X|+1\big|\right] \leq \frac{t^2}{2}\big\||X|-1\big\|_2\big\||X|+1\big\|_2,
\end{align*}
using the Cauchy-Schwarz inequality in the last estimate.
Moreover,
\[
\big\||X|+1\big\|_2 \leq \big\||X|-1\big\|_2 + 2.
\]
Plugging in the assumption $\big\||X|-1\big\|_2 \leq \delta_0$ completes the proof.
\end{proof}


\subsection{Uniform bounds on the special function and its derivative}

\begin{lemma}\label{lm:Psi-unif}
Assuming \eqref{eq:assump-S} and the symmetry of $X_1$, the functions $\Psi$ and $\Psi_0$ defined in \eqref{eq:defPsi} and \eqref{eq:defPsi0} respectively satisfy
\begin{equation}
|\Psi(s) - \Psi_0(s)| \leq \frac{2}{\pi}\sqrt{2\delta_0(\delta_0+2)}, \qquad s \geq 1.
\end{equation}
\end{lemma}
\begin{proof}
Fix $T>0$. Breaking the integral defining $\Psi$ into $\int_0^T + \int_T^\infty$ and using that $|a-b| \leq 1$ for $a, b \in [0,1]$, we obtain
\begin{align*}
|\Psi(s) - \Psi_0(s)| &= \frac{2}{\pi}\left|\int_0^\infty \left[\left|\phi\left(\frac{t}{\sqrt{s}}\right)\right|^s - \left|\cos\left(\frac{t}{\sqrt{s}}\right)\right|^s\right]t^{-2}\dd t\right| \\
&\leq \frac{2}{\pi}\int_0^T \left|\left|\phi\left(\frac{t}{\sqrt{s}}\right)\right|^s - \left|\cos\left(\frac{t}{\sqrt{s}}\right)\right|^s\right|t^{-2}\dd t + \frac{2}{\pi}\int_T^\infty t^{-2}\dd t
\end{align*}
We also have $\big||a|^s - |b|^s\big| \leq s|a-b|$ for $a, b \in [-1,1]$, $s \geq 1$, thus Lemma \ref{lm:phi-unif} yields
\begin{align*}
|\Psi(s) - \Psi_0(s)| &\leq \frac{2}{\pi}\int_0^T s\frac{\delta_0(\delta_0+2)}{2}\left(\frac{t}{\sqrt{s}}\right)^2t^{-2}\dd t + \frac{2}{\pi T} = \frac{2}{\pi}\left(T\frac{\delta_0(\delta_0+2)}{2} + \frac{1}{T}\right).
\end{align*}
Optimizing over the parameter $T$ gives the desired bound.
\end{proof}

\begin{lemma}\label{lm:sulogu}
For $s \geq 2$ and $0 < u, v < 1$, we have
\[
|u^s\log u - v^s\log v| \leq |u-v|.
\]
\end{lemma}
\begin{proof}
Let $f(x)=x^s \log x$. It suffices to prove that on $(0,1)$ we have $|f'(x)| \leq 1$, which is equivalent to $|\alpha t \log t + t| \leq 1$ with $t = x^{s-1} \in (0,1)$ and $\alpha = \frac{s}{s-1} \in [1,2]$. To prove this observe that for $t\in(0,1)$ we have $\alpha t \log t + t \leq t \leq 1$ and 
\[
\alpha t \log t + t \geq \alpha t \log t  \geq -\frac{\alpha}{e} \geq -\frac{2}{e} > -1.  \qedhere
\]
\end{proof}

\begin{lemma}\label{lm:DerPsi-unif}
Assuming \eqref{eq:assump-S} and the symmetry of $X_1$, the functions $\Psi$ and $\Psi_0$ defined in \eqref{eq:defPsi} and \eqref{eq:defPsi0} satisfy
\begin{equation}
|\Psi'(s) - \Psi_0'(s)| \leq 0.62\sqrt{\delta_0(\delta_0+2)}, \qquad s \geq 2.
\end{equation}
\end{lemma}

\begin{proof}
Changing the variables and differentiating gives
\begin{align*}
\Psi'(s) &= \frac{\dd}{\dd s}\left(\frac{2}{\pi\sqrt{s}}\int_0^\infty \Big[1-|\phi(t)|^s\Big]t^{-2}\dd t\right) \\
&= -\frac{1}{2s}\Psi(s) - \frac{2}{\pi \sqrt{s}}\int_0^\infty|\phi(t)|^s\log|\phi(t)|t^{-2}\dd t.
\end{align*}
Thus,
\begin{align*}
|\Psi'(s) - \Psi_0'(s)| &\leq \frac{1}{2s}|\Psi(s) - \Psi_0(s)| \\
&\qquad+ \frac{2}{\pi\sqrt{s}}\int_0^\infty\Big||\phi(t)|^s\log|\phi(t)| - |\cos(t)|^s\log|\cos(t)|\Big|t^{-2}\dd t.
\end{align*}
To estimate the integral, we proceed along the same lines as in the proof of Lemma \ref{lm:Psi-unif}. We fix $T  > 0$, write $\int_0^\infty = \int_0^T + \int_T^\infty$ and for the second integral use $|u^s\log u| = \frac{1}{s}|u^s\log (u^s)| \leq \frac{1}{es}$, $0 < u < 1$, to get a bound on it by $\frac{2}{esT}$,
whilst for the first integral, using first Lemma \ref{lm:sulogu} and then Lemma \ref{lm:phi-unif}, we obtain
\begin{align*}
\int_0^T\Big||\phi(t)|^s\log|\phi(t)| - |\cos(t)|^s\log|\cos(t)|\Big|t^{-2}\dd t &\leq \int_0^T |\phi(t) - \cos(t)|t^{-2}\dd t \\
&\leq {\red \frac{\delta_0(\delta_0+2)}{2}T.}
\end{align*}
Altogether, with the aid of Lemma \ref{lm:Psi-unif},
\begin{align*}
|\Psi'(s) - \Psi_0'(s)| &\leq \frac{1}{2s}\frac{2}{\pi}\sqrt{2\delta_0(\delta_0+2)} + \frac{2}{\pi\sqrt{s}}\left(\frac{\delta_0(\delta_0+2)}{2}T + \frac{2}{esT}\right).
\end{align*}
Minimising the second term over $T > 0$ leads to the bound by
\[
\frac{1}{\pi s}\sqrt{2\delta_0(\delta_0+2)} + \frac{4}{\pi s}\sqrt{\frac{\delta_0(\delta_0+2)}{e}} = \frac{\sqrt{\delta_0(\delta_0+2)}}{\pi s}\left(\sqrt{2}+\frac{4}{\sqrt{e}}\right).
\]
For $s \geq 2$, we have $\frac{1}{\pi s}\left(\sqrt{2}+\frac{4}{\sqrt{e}}\right) < 0.61...$ and this completes the proof.
\end{proof}


\subsection{Proof of Lemma \ref{lm:Psi2}}

First we assume that $s \geq 3$. Using Lemma \ref{lm:Psi-unif} and letting $\eta = \frac{2}{\pi}\sqrt{2\delta_0(\delta_0+2)}$ for brevity, we get
\[
\Psi(s) \geq \Psi_0(s) - \eta.
\]
Since $\Psi_0$ is increasing, $\Psi_0(s) \geq \Psi_0(3) = \Psi_0(3) - \Psi_0(2) + \Psi_0(2)$ and $\Psi_0(2) \geq \Psi(2) -\eta$, again using Lemma \ref{lm:Psi-unif}. Therefore,
\[
\Psi(s) \geq \Psi(2) + \big(\Psi_0(3) - \Psi_0(2) -2\eta\big).
\]
It is now clear that as long as $\delta_0$ is sufficiently small, namely $2\eta \leq \Psi_0(3) - \Psi_0(2)$, we get $\Psi(s) \geq \Psi(2)$, as desired. It can be checked that $\Psi_0(3) - \Psi_0(2) = \frac{4}{\pi\sqrt{3}} - \frac{1}{\sqrt{2}} = 0.027..$ and a choice of $\delta_0 \leq  10^{-4}$ suffices for the estimate $\Psi(s)\geq\Psi(2)$ to hold for $s\geq3$.

Now we assume that $2 < s < 3$. We have
\[
\Psi(s) = \Psi(2) + (s-2)\Psi'(\theta)
\]
for some $2 < \theta < s$. Using Lemmas \ref{lm:DerPsi-unif} and \ref{lm:Psi1-bounds}, we get
\[
\Psi'(\theta) \geq \Psi_0'(\theta) - 0.62\sqrt{\delta_0(\delta_0+2)} \geq 0.017 - 0.62\sqrt{\delta_0(\delta_0+2)}
\]
which is positive for all $\delta_0 \leq 3.7\cdot 10^{-4}$. Thus, $\Psi(s) \geq \Psi(2)$ holds in both cases.\hfill$\square$


\section{Proof of Theorem \ref{thm:mainB}}

The approach is the same as for Theorem \ref{eq:Sza}, however certain technical details are substantially more involved. We begin with a Fourier-analytic representation for \emph{negative} moments due to Gorin and Favorov \cite{GF}.

\begin{lemma}[Lemma 3 in \cite{GF}]\label{lm:formula-mom}
For a random vector $X$ in $\R^d$ and $-d < q < 0$, we have
\begin{equation}
\E|X|^{q} = \beta_{q,d}\int_{\R^d}\E e^{i\scal{t}{X}} \cdot |t|^{-q-d} \dd t,
\end{equation}
where $\beta_{q,d} = 2^{q}\pi^{-d/2}\frac{\Gamma\left((d+q)/2\right)}{\Gamma(-q/2)}$, provided that the integral on the right hand side exists.
\end{lemma}

Specialised to $d = 3$, $q = -1$ ($\beta_{-1,3} = \frac{1}{2\pi^2}$) and $X = \sum_{j=1}^n a_jX_j$ with $X_1,\ldots,X_n$ independent random vectors, we obtain
\begin{equation}
\E\left|\sum_{j=1}^n a_jX_j\right|^{-1} = \frac{1}{2\pi^2}\int_{\R^3} \left(\prod_{j=1}^n \E e^{i\scal{t}{a_jX_j}}\right)|t|^{-2} \dd t.
\end{equation}
Note that thanks to the decay assumption \eqref{eq:assump-decay}, the integral on the right hand side converges as long as $n \geq 2$ (assuming the $a_j$ are nonzero). As in Ball's proof from \cite{Ba}, H\"older's inequality yields
\begin{equation}
\E\left|\sum_{j=1}^n a_jX_j\right|^{-1} \leq \prod_{j=1}^n \Phi\big(a_j^{-2}\big)^{a_j^2},
\end{equation}
where
\begin{equation}\label{eq:defPhi}
\Phi(s) = \frac{1}{2\pi^2}\int_{\R^3} \left|\phi\left(s^{-1/2}t\right)\right|^s|t|^{-2} \dd t, \qquad s > 1
\end{equation}
with
\begin{equation}
\phi(t) = \E e^{i\scal{t}{X_1}}, \qquad t \in \R^3,
\end{equation}
denoting the characteristic function of $X_1$. Exactly as in the proof of Theorem \ref{eq:Sza}, the following pivotal lemma allows us to finish the proof.

\begin{lemma}\label{lm:Phi2}
Under the assumptions of Theorem \ref{thm:mainB},  we have $\Phi(s) \leq \Phi(2)$ for every $s\geq2$.
\end{lemma}

If the $X_j$ are uniform on the unit sphere $\mathbb{S}^2$ in $\R^3$, we have $\phi(t) = \frac{\sin|t|}{|t|}$ (because $\scal{t}{X_1}$ is uniform on $[-|t|,|t|]$), in which case the special function $\Phi$ defined in \eqref{eq:defPhi} becomes
\begin{equation}\label{eq:defPhi0}
\Phi_0(s) = \frac{2}{\pi}\int_0^\infty \left|\frac{\sin(s^{-1/2}t)}{s^{-1/2}t}\right|^s\dd t, \qquad s > 1
\end{equation}
(after integrating in polar coordinates). Ball's celebrated integral inequality states that $\Phi_0(s) \leq \Phi_0(2)$, for all $s \geq 2$ (see Lemma 3 in \cite{Ba}, as well as \cite{MR-maj, NP} for different proofs). Our proof of Lemma \ref{lm:Phi2} relies on this, additional bounds on the derivative $\Phi_0'(s)$ near $s=2$, as well as, crucially, bounds quantifying how close $\Phi$ is to $\Phi_0$. In the following subsections we gather such results and then conclude with the proof of Lemma \ref{lm:Phi2}.


\subsection{A uniform bound on the characteristic function}

Throughout these sections $\xi$ always denotes a random vector uniform on the unit sphere $\mathbb{S}^2$ in $\R^3$.

\begin{lemma}\label{lm:phi-unif-vec}
Let $X$ be a symmetric random vector in $\R^3$ with $\delta = \msf{W}_2(X,\xi)$. Then, its characteristic function $\phi(t) = \E e^{i\scal{t}{X}}$ satisfies
\begin{equation}
\left|\phi(t) - \frac{\sin|t|}{|t|}\right| \leq \frac{\delta(\delta+2)}{2}|t|^2, \qquad t \in \R^3.
\end{equation}
\end{lemma}
\begin{proof}
Let $\xi$ be uniform on $\mathbb{S}^2$ such that for the joint distribution of $(X, \xi)$, we have $\|X-\xi\|_2 = \msf{W}_2(X,\xi) = \delta$. By symmetry, the bound $|\sin u| \leq |u|$ and the Cauchy-Schwarz inequality (used twice), we get
\begin{align*}
\left|\phi(t) - \frac{\sin|t|}{|t|}\right| 
&= \left|\E\left[\cos\scal{t}{X}-\cos\scal{t}{\xi}\right]\right| \\
&= 2\left|\E\left[\sin\left(\tfrac12\scal{t}{X-\xi}\right)\sin\left(\tfrac12\scal{t}{X+\xi}\right)\right]\right| \\
&\leq \frac{|t|^2}{2}\E\left[\big|X-\xi\big|\cdot\big|X+\xi\big|\right] \\
&\leq \frac{|t|^2}{2}\big\|X-\xi\big\|_2\big\|X+\xi\big\|_2.
\end{align*}
To conclude we use the triangle inequality
\[
\big\|X+\xi\big\|_2 \leq \big\|X-\xi\big\|_2 + 2\|\xi\|_2 = \big\|X-\xi\big\|_2 +2.\qedhere
\]
\end{proof}


\subsection{Bounds on the special function}

We begin with a bound on the difference $\Phi(s) - \Phi_0(s)$ obtained from the uniform bound on the characteristic functions (Lemma \ref{lm:phi-unif-vec} above). In contrast to Lemma \ref{lm:Psi-unif}, the bound is not uniform in~$s$. For $s$ \emph{not} too large (the bulk), we incur the factor $s^{3/4}$. To fight it off for large values of $s$, we shall employ a Gaussian approximation. For that part to work, it is crucial that $\Phi_0(2) - \Phi_0(\infty) = \sqrt{2} - \sqrt{\frac{6}{\pi}} > 0$.


\subsubsection{The bulk}

\begin{lemma}\label{lm:Phi-bulk}
Let $X$ be a symmetric random vector in $\R^3$ with $\delta = \msf{W}_2(X,\xi)$ and characteristic function $\phi$ satisfying \eqref{eq:assump-decay} for some $C_0 > 0$. Let $\Phi$ and $\Phi_0$ be defined through \eqref{eq:defPhi} and \eqref{eq:defPhi0} respectively.  For every $s \geq 2$, we have
\begin{equation}
|\Phi(s) - \Phi_0(s)| \leq \frac{2^{11/4}}{3\pi}s^{3/4}\big(\delta(\delta+2)\big)^{1/4}\big(C_0^2+1\big)^{3/4}.
\end{equation}
\end{lemma}
\begin{proof}
Given the definitions, we have
\[
\Phi(s) - \Phi_0(s) = \frac{\sqrt{s}}{2\pi^2}\int_{\R^3} \left(|\phi(t)|^s - \left|\frac{\sin|t|}{|t|}\right|^s\right)|t|^{-2} \dd t.
\]
We fix $T>0$ and split the integration into two regions.

\noindent \emph{Small $t$.} Using Lemma \ref{lm:phi-unif-vec} and $||a|^s-|b|^s| \leq s|a-b|$ when $|a|, |b| \leq 1$, we obtain
\[
\left|\int_{|t|\leq T} \left(|\phi(t)|^s - \left|\frac{\sin|t|}{|t|}\right|^s\right)|t|^{-2} \dd t\right| \leq s\frac{\delta(\delta+2)}{2}\int_{|t|\leq T} \dd t = \frac{2\pi}{3} s\delta(\delta+2)T^3.
\]

\noindent \emph{Large $t$.} Since $s \geq 2$, we have
\[
\left|\int_{|t|\geq T} \left(|\phi(t)|^s - \left|\frac{\sin|t|}{|t|}\right|^s\right)|t|^{-2} \dd t\right| \leq \int_{|t|\geq T} \left(|\phi(t)|^2 + \left|\frac{\sin|t|}{|t|}\right|^2\right)|t|^{-2} \dd t.
\]
By virtue of the decay assumption \eqref{eq:assump-decay}, this is at most
\[
\int_{|t|\geq T} \frac{C_0^2+1}{|t|^4} \dd t = {\red 4\pi\frac{C_0^2+1}{T}.}
\]
Adding up these two bounds and optimising over $T$ yields
\[
\left|\int_{\R^3} \left(|\phi(t)|^s - \left|\frac{\sin|t|}{|t|}\right|^s\right)|t|^{-2} \dd t\right| \leq  \frac{2^{15/4}\pi}{3}s^{1/4}\big(\delta(\delta+2)\big)^{1/4}\big(C_0^2+1\big)^{3/4}.
\]
Plugging this back gives the assertion.
\end{proof}


\subsubsection{The Gaussian approximation} We now present a bound on $\Phi(s)$ which does not grow as $s\to\infty$ that will allow us to prove Lemma \ref{lm:Phi2} for $s$ sufficiently large.

\begin{lemma}\label{lm:Phi-big-s}
Let $X$ be a symmetric random vector in $\R^3$ with $\delta = \msf{W}_2(X,\xi)$ and characteristic function $\phi$ satisfying \eqref{eq:assump-decay} for some $C_0 > 0$. Let $\Phi$ be defined through \eqref{eq:defPhi}. Assuming that $\delta \leq \min\{\frac{1}{\sqrt{3}}, (15C_0)^{-2}\}$, we have
\begin{equation}
\begin{split}
\Phi(s) \leq &\sqrt{\frac{6}{\pi}}\Big((1-\delta\sqrt{3})^2-\theta\E|X|^3\Big)^{-1/2} \\
 &+ \sqrt{\frac{6}{\pi}}\exp\left\{-s\left(\frac{\theta^2}{6}-26\delta(\delta+2)\right)\right\}+ 2C_0\left(\sqrt{s}+\frac{2}{\sqrt{s}}\right)e^{-s}, \quad s \geq 2,
 \end{split}
\end{equation}
with arbitrary $0 < \theta < \frac{(1-\delta\sqrt{3})^2}{3\E|X|^3}$.
\end{lemma}
\begin{proof}
We split the integral defining $\Phi(s) = \frac{1}{2\pi^2}\int_{\R^3} |\phi(s^{-1/2}t)|^s|t|^{-2} \dd t$ into several regions.

\noindent \emph{Large $t$.} Using the decay condition \eqref{eq:assump-decay}, we get
\[
\int_{|t| \geq eC_0\sqrt{s}} \left|\phi\left(s^{-1/2}t\right)\right|^s|t|^{-2} \dd t \leq \int_{|t| \geq eC_0\sqrt{s}} C_0^s|s^{-1/2}t|^{-s}|t|^{-2} \dd t = \frac{4\pi e\sqrt{s}}{s-1}C_0e^{-s}.
\]
Thus, for $s \geq 2$,
\[
\frac{1}{2\pi^2}\int_{|t| \geq eC_0\sqrt{s}} \left|\phi\left(s^{-1/2}t\right)\right|^s|t|^{-2} \dd t \leq \frac{2 e\sqrt{s}}{\pi(s-1)}C_0e^{-s} < \frac{4C_0}{\sqrt{s}}e^{-s},
\]
as $\frac{2 e\sqrt{s}}{\pi(s-1)} < \frac{4}{\sqrt{s}}$ for $s \geq 2$.

\smallskip 
\noindent \emph{Moderate $t$.} This case is vacuous unless $C_0 > \pi/e$. We use Lemma \ref{lm:phi-unif-vec} to obtain
\begin{align*}
&\int_{\pi\sqrt{s} \leq |t| \leq eC_0\sqrt{s}} \left|\phi\left(s^{-1/2}t\right)\right|^s|t|^{-2} \dd t  \\
&\leq \int_{\pi\sqrt{s} \leq |t| \leq eC_0\sqrt{s}} \left(\left|\frac{\sin(s^{-1/2}|t|)}{s^{-1/2}|t|} \right|+\frac{\delta(\delta+2)}{2}\left(s^{-1/2}|t|\right)^2\right)^s|t|^{-2} \dd t \\
&\leq \int_{\pi\sqrt{s} \leq |t| \leq eC_0\sqrt{s}} \left(\frac{1}{\pi}+\frac{\delta(\delta+2)}{2}\left(eC_0\right)^2\right)^s|t|^{-2} \dd t \\
&= 4\pi\sqrt{s}\left(\frac{1}{\pi}+\frac{\delta(\delta+2)}{2}\left(eC_0\right)^2\right)^s (eC_0-\pi)_+.
\end{align*}
In this case,  the condition $\delta < (15C_0)^{-2}$ suffices to guarantee that $\frac{1}{\pi}+\frac{\delta(\delta+2)}{2}\left(eC_0\right)^2 < \frac{1}{e}$ (also using, say $\delta +2 < 3$). Then we get
\[
\frac{1}{2\pi^2}\int_{\pi\sqrt{s} \leq |t| \leq eC_0\sqrt{s}} \left|\phi\left(s^{-1/2}t\right)\right|^s|t|^{-2} \dd t \leq \frac{2}{\pi}\sqrt{s}e^{-s}(eC_0-\pi)_+ < 2C_0\sqrt{s}e^{-s}.
\]

\smallskip

\noindent \emph{Small $t$.} For $0 < u < \pi$, we have 
\begin{equation}\label{eq:sinc}
\frac{\sin u}{u} =  \prod_{k=1}^\infty \left(1-\frac{u^2}{(k\pi)^2}\right) \leq \exp\Big(-\sum_{k=1}^\infty \frac{u^2}{(k\pi)^2}\Big) =  e^{-u^2/6}.
\end{equation}
Fix $0 < \theta < \pi$. Then, first using Lemma \ref{lm:phi-unif-vec} and then \eqref{eq:sinc}, we obtain
\begin{align*}
&\int_{\theta\sqrt{s} \leq |t| \leq \pi\sqrt{s}} \left|\phi\left(s^{-1/2}t\right)\right|^s|t|^{-2} \dd t  \\
&\leq \int_{\theta\sqrt{s} \leq |t| \leq \pi\sqrt{s}} \left(\left|\frac{\sin(s^{-1/2}|t|)}{s^{-1/2}|t|} \right|+\frac{\delta(\delta+2)}{2}\left(s^{-1/2}|t|\right)^2\right)^s|t|^{-2} \dd t \\
&\leq \int_{\theta\sqrt{s} \leq |t| \leq \pi\sqrt{s}} \left(e^{-|t|^2/(6s)}+\frac{\delta(\delta+2)}{2}\pi^2\right)^s|t|^{-2} \dd t \\
&\leq \int_{|t| \geq \theta\sqrt{s}} e^{-|t|^2/6}\left(1+\frac{\delta(\delta+2)}{2}\pi^2e^{\pi^2/6}\right)^s|t|^{-2} \dd t.
\end{align*}
Integrating using polar coordinates and invoking the standard tail bound 
\[
\int_u^\infty e^{-y^2/2} \dd y \leq \sqrt{\pi/2}e^{-u^2/2}, \qquad u > 0,
\]
the last integral gets upper bounded by
\[
4\pi^{3/2}\sqrt{\frac{3}{2}}e^{-\theta^2s/6}\left(1+\frac{\delta(\delta+2)}{2}\pi^2e^{\pi^2/6}\right)^s < 4\pi^{3/2}\sqrt{\frac{3}{2}}e^{-\theta^2s/6}{\red \big(1+26\delta(\delta+2)\big)^s.}
\]
Summarising, we have shown that
\begin{align*}
\frac{1}{2\pi^2}\int_{\theta\sqrt{s} \leq |t| \leq \pi\sqrt{s}} \left|\phi\left(s^{-1/2}t\right)\right|^s|t|^{-2} \dd t  &\leq \sqrt{\frac{6}{\pi}}\big(1+26\delta(\delta+2)\big)^se^{-s\theta^2/6} \\
&\leq \sqrt{\frac{6}{\pi}}\exp\left\{-s\left(\frac{\theta^2}{6}-26\delta(\delta+2)\right)\right\}.
\end{align*}

\smallskip

\noindent \emph{Very small $t$.} 
Taylor-expanding $\phi$ at $0$ with the Lagrange remainder,
\begin{align*}
&\int_{|t| \leq \theta\sqrt{s}} \left|\phi\left(s^{-1/2}t\right)\right|^s|t|^{-2} \dd t  \\
&= \int_{|t| \leq \theta\sqrt{s}} \left|1 - \frac{1}{2}\E\scal{X}{s^{-1/2}t}^2 + \frac{s^{-3/2}}{6}\sum_{j,k,l=1}^3\frac{\partial^3\phi}{\partial t_j\partial t_k \partial t_l}({\red \eta})t_jt_kt_l\right|^s \dd t,
\end{align*}
for some point {\red $\eta$} in the segment $[0,s^{-1/2}t]$. To bound the error term, we note that
\[
\left|\frac{\partial^3\phi}{\partial t_j\partial t_k \partial t_l}({\red \eta})\right| \leq \E|X_jX_kX_l|,
\]
thus
\[
\left|\sum_{j,k,l=1}^3\frac{\partial^3\phi}{\partial t_j\partial t_k \partial t_l}({\red \eta})t_jt_kt_l\right| \leq \E\left(|t_1||X_1| + |t_2||X_2| + |t_3||X_3|\right)^3 \leq |t|^3\E|X|^3.
\]
We also note that in the domain $\{|t| \leq \theta\sqrt{s}\}$, the leading term $1 - \frac12\E\scal{X}{s^{-1/2}t}^2$ is nonnegative, provided that $\frac12\theta^2\E|X|^2 \leq 1$. Since $\|X\|_2 \leq \delta + 1$ under the assumption \eqref{eq:assump-B}, it suffices that $\theta < \frac{\sqrt{2}}{1+\delta}$. Assuming this, we thus get
\begin{align*}
&\int_{|t| \leq \theta\sqrt{s}} \left|\phi\left(s^{-1/2}t\right)\right|^s|t|^{-2} \dd t 
\leq \int_{|t| \leq \theta\sqrt{s}} \left(1 - \frac{1}{2}\E\scal{X}{s^{-1/2}t}^2 + \frac16|s^{-1/2}t|^3\E|X|^3 \right)^s |t|^{-2}\dd t.
\end{align*}
Evoking \eqref{eq:assump-B}, let $\xi$ be uniform on $\mathbb{S}^2$ such that $\|X-\xi\|_2 \leq \delta$ with respect to some coupling. Then, for a fixed vector $v$ in $\R^3$, we obtain the bound
\[
\left\|\scal{X}{v}\right\|_2 \geq \|\scal{\xi}{v}\|_2 - \|\scal{X-\xi}{v}\|_2 = \tfrac{1}{\sqrt{3}}|v|- \|\scal{X-\xi}{v}\|_2 \geq \tfrac{1}{\sqrt{3}}|v| - \delta|v|.
\]
Thus, provided that $\delta < \frac{1}{\sqrt{3}}$, this yields
\begin{align*}
\int_{|t| \leq \theta\sqrt{s}} \left|\phi\left(s^{-1/2}t\right)\right|^s|t|^{-2} \dd t & \leq \int_{|t| \leq \theta\sqrt{s}} \left(1 - \frac{(1/\sqrt{3}-\delta)^2}{2s}|t|^2 + \frac{\theta\E|X|^3}{6s}|t|^2 \right)^s |t|^{-2}\dd t \\
&\leq \int_{\R^3} \exp\left(-\alpha|t|^{2}/2\right)|t|^{-2}\dd t = \frac{2\pi\sqrt{2\pi}}{\sqrt{\alpha}},
\end{align*}
where we have set $\alpha = (\frac{1}{\sqrt{3}}-\delta)^2 - \frac13\theta\E|X|^3$ and assumed that $\alpha$ is positive in the last equality (guaranteed by choosing $\theta$ sufficiently small). Then we finally obtain 
\[
\frac{1}{2\pi^2}\int_{|t| \leq \theta\sqrt{s}} \left|\phi\left(s^{-1/2}t\right)\right|^s|t|^{-2} \dd t \leq \sqrt{\frac{2}{\pi\alpha}}.
\]
Putting these three bounds together gives the assertion. Note that we have imposed the conditions $\delta < \frac{1}{\sqrt{3}}$ and $\delta < (15C_0)^{-2}$ when $C_0 > \frac{\pi}{e}$, as well as $\theta < \pi$, $\theta < \frac{\sqrt{2}}{1+\delta}$ and $\theta < \frac{(1-\delta\sqrt{3})^2}{3\E|X|^3}$. Since $\|X\|_3 \geq \|X\|_2 \geq 1 - \delta$ and $\delta < \frac{1}{\sqrt{3}}$, we have $\frac{(1-\delta\sqrt{3})^2}{3\E|X|^3} < \frac{(1-\delta\sqrt{3})^2}{3(1-\delta)^3} = \frac{1}{3(1-\delta)}\left(\frac{1-\delta\sqrt{3}}{1-\delta}\right)^2 < \frac{1}{3-\sqrt{3}} < 0.79$. Moreover, $\frac{\sqrt{2}}{1+\delta} > \frac{\sqrt{2}}{1+1/\sqrt{3}} > 0.89$, so the condition $\theta < \frac{(1-\delta\sqrt{3})^2}{3\E|X|^3}$ implies the other two conditions on $\theta$.
\end{proof}


\subsection{Bounds on the derivative of the special function}

\begin{lemma}\label{lm:Phi-der}
Let $X$ be a symmetric random vector in $\R^3$ with $\delta = \msf{W}_2(X,\xi)$ and characteristic function $\phi$ satisfying \eqref{eq:assump-decay} for some $C_0 > 0$. Let $\Phi$ and $\Phi_0$ be defined through \eqref{eq:defPhi} and \eqref{eq:defPhi0} respectively.   For every $s \geq 2$, we have
\begin{align*}
|\Phi'(s) - \Phi_0'(s)| &\leq \frac{2^{7/4}}{3\pi}\big(\delta(\delta+2)\big)^{1/4}\big(C_0^2+1\big)^{3/4}s^{-1/4} +1.04 \big(\delta(\delta+2)\big)^{1/7}\big(C_0^{3/2}+1\big)^{6/7}s^{1/2}.
\end{align*}
\end{lemma}
\begin{proof}
First we take the derivative,
\[
\Phi'(s) = \frac{\dd}{\dd s} \left(\frac{\sqrt{s}}{2\pi^2}\int_{\R^3} |\phi(t)|^s \dd t\right) = \frac{1}{2s}\Phi(s) + \frac{\sqrt{s}}{2\pi^2}\int_{\R^3} |\phi(t)|^s\log|\phi(t)| \dd t.
\]
For the resulting $\Phi-\Phi_0$ term, we use Lemma \ref{lm:Phi-bulk}. To bound the difference of the integrals resulting from the second term, we fix $T>0$ and split the integration into two regions.

\noindent \emph{Small $t$.} Using Lemmas \ref{lm:sulogu} and \ref{lm:phi-unif-vec}, we obtain
\begin{align*}
&\left|\int_{|t| \leq T} \left(|\phi(t)|^s\log|\phi(t)| - \left|\frac{\sin|t|}{|t|}\right|^s\log\left|\frac{\sin|t|}{|t|}\right|\right)|t|^{-2} \dd t\right| \\
&\leq \int_{|t|\leq T} \frac{\delta(\delta+2)}{2}\dd t = \frac{2\pi}{3} \delta(\delta+2)T^3.
\end{align*}

\noindent \emph{Large $t$.} Note that for $s \geq 2$, and $0 < u < 1$ we have,
\[
|u^s \log u| = |2u^{s-1/2}u^{1/2}\log(u^{1/2})| \leq \frac{2}{e}u^{3/2}.
\]
Thus,
\begin{align*}
&\left|\int_{|t| \geq T} \left(|\phi(t)|^s\log|\phi(t)| - \left|\frac{\sin|t|}{|t|}\right|^s\log\left|\frac{\sin|t|}{|t|}\right|\right)|t|^{-2} \dd t\right| \\
&\leq \frac{2}{e}\left|\int_{|t| \geq T} \left(|\phi(t)|^{3/2} + \left|\frac{\sin|t|}{|t|}\right|^{3/2}\right)|t|^{-2} \dd t\right|
\end{align*}
which, after applying the decay condition \eqref{eq:assump-decay}, gets upper bounded by
\[
\frac{8\pi}{e}\int_T^\infty \frac{C_0^{3/2}+1}{t^{3/2}}\dd t = \frac{16\pi}{e}(C_0^{3/2}+1)T^{-1/2}.
\]
Adding up these two bounds and optimising over $T$ yields
\begin{align*}
&\left|\int_{\R^3} \left(|\phi(t)|^s\log|\phi(t)| - \left|\frac{\sin|t|}{|t|}\right|^s\log\left|\frac{\sin|t|}{|t|}\right|\right)|t|^{-2} \dd t\right| \\
&\leq \frac{7\cdot 2^{19/7}\pi}{3e^{6/7}}\big(\delta(\delta+2)\big)^{1/7}\big(C_0^{3/2}+1\big)^{6/7}.
\end{align*}
Going back to the difference of the derivatives, we arrive at the desired bound using
\[
\frac{7\cdot 2^{12/7}}{3e^{6/7}\pi} < 1.04. \qedhere
\]
\end{proof}


\subsection{Bounds on Ball's special function}

We will need two estimates on $\Phi_0$ defined in \eqref{eq:defPhi0}, that is
\begin{equation}
\Phi_0(s) = \frac{2}{\pi}\int_0^\infty \left|\frac{\sin(s^{-1/2}t)}{s^{-1/2}t}\right|^s\dd t = \frac{2\sqrt{s}}{\pi}\int_0^\infty \left|\frac{\sin t}{t}\right|^s \dd t, \qquad s > 1.
\end{equation}

First, we have a bound on the derivative near $s=2$.

\begin{lemma}\label{lm:Phi0-der}
For $2 \leq s \leq 2.01$, we have $\Phi_0'(s) \leq -0.02$. 
\end{lemma}

Second, on the complementary range, $\Phi_0(s)$ is separated from its supremal value $\Phi_0(2)$.

\begin{lemma}\label{lm:Phi0}
For $s \geq 2.01$, we have 
$\Phi_0(s) \leq \Phi_0(2) - 2\cdot 10^{-4}.$
\end{lemma}

We begin with a numerical bound which will be used in the proofs of these assertions.

\begin{lemma}\label{lm:sec-der-2}
We have
\[
	\int_0^\infty \left( \frac{\sin u}{u} \right)^2 \log\left| \frac{\sin u}{u} \right|  \dd u \leq -0.48.
\]
\end{lemma}

\begin{proof}
Using \eqref{eq:sinc}, we get
\[ 
\int_0^\pi \left( \frac{\sin u}{u} \right)^2 \log\left| \frac{\sin u}{u} \right|  \dd u \leq -\frac16 \int_0^\pi (\sin u)^2  \dd u = -\frac{\pi}{12}.
\]
Moreover,
\begin{align*}
	\int_\pi^\infty \left( \frac{\sin u}{u} \right)^2&\log\left| \frac{\sin u}{u} \right|  \dd u  = \sum_{k=1}^\infty \int_{k \pi}^{(k+1)\pi} \left( \frac{\sin u}{u} \right)^2 \log\left| \frac{\sin u}{u} \right| \\
	& \leq \sum_{k=1}^\infty \int_{k \pi}^{(k+1)\pi} \left( \frac{\sin u}{(k+1) \pi} \right)^2 \log\left| \frac{1}{k \pi} \right|  = -\frac{1}{2\pi} \sum_{k=1}^\infty \frac{\log(k \pi)}{(k+1)^2}.
\end{align*}
Therefore our integral is bounded above by
\[
-\frac{\pi}{12} - \frac{1}{2\pi} \sum_{k=1}^\infty \frac{\log(k \pi)}{(k+1)^2} = -0.4867.. < -0.48.\qedhere
\]
\end{proof}

We let
\begin{equation}
I(s)=\int_0^\infty \left| \frac{\sin u}{u} \right|^s \dd u, \qquad s > 1.
\end{equation}

\begin{proof}[Proof of Lemma \ref{lm:Phi0-der}]
First we observe that 
\[
I'(s)=\int_0^\infty \left| \frac{\sin u}{u} \right|^s \log\left| \frac{\sin u}{u} \right| \dd u.
\]
Note that $I$ is decreasing. We have,
\[
	\Phi_0'(s) = \frac{2}{\pi}\left(\frac{I(s)}{2\sqrt{s}} + \sqrt{s} I'(s)\right)  \leq \frac{2}{\pi}\left(\frac{I(2)}{2\sqrt{s}} + \sqrt{s} I'(s)\right) = \frac{1}{2\sqrt{s}} + \frac{2\sqrt{s}}{\pi} I'(s),
\] 
since $I(2)=\frac{\pi}{2}$.  Moreover,
\begin{align*}
	|I''(s)| & = \int_0^\infty \left| \frac{\sin u}{u} \right|^s \log^2\left| \frac{\sin u}{u} \right| \dd u \leq \int_0^\infty \left| \frac{\sin u}{u} \right|^2 \log^2\left| \frac{\sin u}{u} \right| \dd u \\
	&  \leq \sup_{t \in (0,1)} (\sqrt{t} \log^2 t) \int_0^\infty \left| \frac{\sin u}{u} \right|^{3/2} \dd u = 16 e^{-2} \int_0^\infty \left| \frac{\sin u}{u} \right|^{3/2} \dd u \\
	& \leq 16 e^{-2} \left( 1 + \int_1^\infty  \frac{1}{u^{3/2}}  \dd u \right) = 48 e^{-2}.
\end{align*}
With the aid of Lemma \ref{lm:sec-der-2}, we therefore have 
\[
	I'(s) \leq I'(2) + 48 e^{-2}(s-2) < -0.48 + 48 e^{-2}(s-2) .
\]
Thus, for $2 \leq s \leq 2.01$, we have
\begin{align*}
	\Phi_0'(s) & \leq \frac{1}{2\sqrt{s}} + \frac{2\sqrt{s}}{\pi} I'(s) < \frac{1}{2\sqrt{s}} + \frac{2\sqrt{s}}{\pi}\big(-0.48 + 48 e^{-2}(s-2)\big) \\
	& < \frac{1}{2\sqrt{2}} + \frac{2\sqrt{2}}{\pi}(-0.48 + 48 e^{-2}(s-2)) \\
	& \leq \frac{1}{2\sqrt{2}} + \frac{2\sqrt{2}}{\pi}(-0.48 + 48 e^{-2}0.01)< -0.02,
\end{align*}
where in the first inequality we used that the term in parenthesis is negative.
\end{proof}

For the proof of Lemma \ref{lm:Phi0}, we need several more estimates. First,  we record a lower bound on the derivative of $\Phi_0(s)$ for arbitrary $s$.

\begin{lemma}\label{lm:Phi0-der-lb}
For $s \geq 2$, we have $\Phi_0'(s) \geq -\frac{12\sqrt{s}}{\pi e}$.
\end{lemma}

\begin{proof}
We have,
\[
\Phi_0'(s) = \frac{2}{\pi}\left(\frac{I(s)}{2\sqrt{s}} + \sqrt{s} I'(s)\right) \geq \frac{2\sqrt{s}}{\pi} I'(s),
\]
so it is enough to upper bound $|I'(s)|$. Note that
\begin{align*}
	|I'(s)| & = \int_0^\infty \left| \frac{\sin u}{u} \right|^s \left( - \log\left| \frac{\sin u}{u} \right| \right) \dd u \\
	&\leq \int_0^\infty \left| \frac{\sin u}{u} \right|^2 \left( - \log\left| \frac{\sin u}{u} \right| \right) \dd u \\
	&\leq \sup_{t \in (0,1)} (-\sqrt{t} \log t) \int_0^\infty  \left| \frac{\sin u}{u} \right|^\frac{3}{2} \dd u  \\
	&\leq  2 e^{-1} \left( 1 +  \int_1^\infty   \frac{1}{u^\frac{3}{2}}  \dd u  \right) = 6 e^{-1}.\qedhere
\end{align*}
\end{proof}

Second,  we obtain a quantitative drop-off of the values of $\Phi_0$.

\begin{lemma}\label{lm:sec-impr-ball}
Let $a \in [1, \frac{\pi}{3}]$ and suppose that for some $s_0 \geq 2$, we have $\Phi_0(s_0)= \sqrt{\frac{2}{a}}$. 
Then 
\begin{equation}
	\Phi_0(s) \leq \sqrt{\frac{2}{a}}, \qquad s \geq s_0.
\end{equation}
\end{lemma}

To prove this, we build on the argument of Nazarov and Podkorytov from \cite{NP}. For a somewhat similar bound, we refer to Proposition 7 in K\"onig and Koldobsky's work \cite{KoKo} on maximal-perimeter sections of the cube. For convenience and completeness, we include all arguments in detail. We consider functions
\begin{equation}
	f_a(x) = e^{-\frac{\pi}{2} x^2 a}, \quad g(x) = \left| \frac{\sin \pi x}{\pi x} \right|,  \qquad x>0,
\end{equation}
and their distribution functions 
\begin{equation}
	F_a(y) = |\{x>0: f_a(x) > y \}|, \quad G(y) = |\{x>0: g(x) > y \}|, \qquad y>0.  
\end{equation}

\begin{lemma}\label{lm:sec-sign-change}
For $a \in [1,\frac{\pi}{3}]$ the function $F_a-G$ has precisely one sign change point $y_0$ and at this point changes sign from $"-"$ to $"+"$.  
\end{lemma}

\begin{proof}
Note that $F_a(y)=G(y)=0$ for $y \geq 1$, so we only consider $y \in (0,1)$. We have $F_a(y)=\sqrt{\frac{2}{\pi a} \ln(\frac1y)}$. 

The function $g(x)$ has zeros for $x \in \Z$. For $m\in\mathbb{N}$, let $y_m = \max_{[m,m+1]} g$. We clearly have $y_m< \frac{1}{\pi m}$ and $y_m> g(m+\frac12) = \frac{1}{\pi(m+\frac12)}$. Thus $y_m \in (\frac{1}{\pi(m+\frac12)}, \frac{1}{\pi m})$, which shows that the sequence $y_m$ is decreasing. We have the following claims. 

\medskip

\noindent \textbf{Claim 1}. The function $F_a-G$ is positive on $(y_1,1)$. 

\medskip

Note that if $g(x)>y_1$ then $x \in (0,1)$. Moreover $g(x) \leq {\red f_a(x)}$ for $x \in [0,1]$, since
\[
	g(x)=\frac{\sin \pi x}{\pi x} = \prod_{k=1}^\infty \left( 1 - \frac{x^2}{k^2} \right) \leq \prod_{k=1}^\infty e^{- \frac{x^2}{k^2}} = e^{- \frac{\pi^2}{6} x^2} \leq e^{-\frac{\pi}{2} a x^2} = f_a(x).
\]
Thus,  for $y\in(y_1,1)$,  we have
\[
	G(y)= |\{x \in (0,1): \ g(x)>y \}| < |\{x \in (0,1): \  f_a(x)>y \}| \leq F_a(y).
\]	

\noindent \textbf{Claim 2}. The function $F_a-G$ changes sign at least once in $(0,1)$. 

Due to Claim 1 it is enough to show that $F_a-G$ is sometimes negative. We have $F_a-G \leq F_1 -G$ and 
$\int_0^\infty 2y(F_1(y)-G(y)) \dd y = \int (f_1^2 -g^2)=0$, so $F_1-G$ can be negative. 

\medskip

\noindent \textbf{Claim 3}. The function $F_a-G$ is increasing on $(0,y_1)$. 

Clearly $F_a'>F_1'$ and thus the claim follows from the fact that $F_1-G$ is increasing on $(0,y_1)$, which was proved in \cite{NP} (Chapter I, Step 5).
\end{proof}

\begin{proof}[Proof of Lemma \ref{lm:sec-impr-ball}]
The assumption $\Phi_0(s_0) = {\red \sqrt{\frac{2}{a}}}$  is equivalent to
\[
	\int_0^\infty \left| \frac{\sin \pi x}{\pi x} \right|^{s_0} \dd x = \int_0^\infty \left|e^{-\frac{\pi}{2} x^2 a} \right|^{s_0} \dd x. 
\]
After changing variables and using Lemma \ref{lm:sec-sign-change},  we get from the Nazarov--Podkorytov lemma (Chapter I, Step 4 in \cite{NP}) that for $s \geq s_0$
\[
	\int_0^\infty \left| \frac{\sin  x}{ x} \right|^{s} \dd x \leq \int_0^\infty \left|e^{-\frac{1}{2 \pi} x^2 a} \right|^{s} \dd x = \frac{\pi}{\sqrt{2a s}}.\qedhere
\]
\end{proof}

\begin{proof}[Proof of Lemma \ref{lm:Phi0}]
Take $s_0=2.01$ and $a=2\Phi_0(s_0)^{-2}$ in Lemma \ref{lm:sec-impr-ball}. Since $\Phi_0(2) = \sqrt{2}$, Ball's inequality gives that $a \geq 1$. We need to check that $a  \leq \frac{\pi}{3}$. From Lemma \ref{lm:Phi0-der-lb}, we have that for $s \in [2,2.01]$, $\Phi_0'(s) \geq -\frac{12\sqrt{2.01}}{\pi e} > -2$. Thus, $\Phi_0(s_0) \geq \Phi_0(2)-2(s_0-2) = \sqrt{2} - 0.02$. Therefore, $a < 2\cdot (\sqrt{2}-0.02)^{-2} < 1.03 < \frac{\pi}{3}$, as needed. By Lemmas \ref{lm:sec-impr-ball} and \ref{lm:Phi0-der}, we thus get that for $s \geq s_0=2.01$,
\[
	\Phi_0(s) \leq \sqrt{\frac{2}{a}} = \Phi_0(s_0) \leq \Phi_0(2) +\sup_{[2,2.01]} \Phi_0'\cdot 0.01 \leq \Phi_0(2) - 0.02\cdot 0.01.\qedhere
\]
\end{proof}


\subsection{Proof of Lemma \ref{lm:Phi2}}

Recall that we assume $X$ is a symmetric random vector in $\R^3$ with $\delta = \msf{W}_2(X, \xi)$ and characteristic function $\phi$ satisfying \eqref{eq:assump-decay}, that is $|\phi(t)| \leq C_0/|t|$, for all $t \in \R^3\setminus\{0\}$. Let $C_1 = \max\{C_0,1\}$. Our goal is to show that if \eqref{eq:assump-B} holds, that is
\[
\delta \leq 10^{-38}C_1^{-9}\min\big\{(\E|X|^3)^{-6},1\big\},
\]
then $\Phi(s) \leq \Phi(2)$ for all $s \geq 2$, where $\Phi$ is defined in \eqref{eq:defPhi}. 
For the sake of clarity, we shall be fairly lavish with choosing constants.  Since $C_1 \geq 1$, the above assumes in particular that $\delta \leq 10^{-38}$. With this in mind, we note the following consequences of Lemmas \ref{lm:Phi-bulk} and \ref{lm:Phi-der} respectively: for $s \geq 2$,
\begin{equation}\label{eq:Phi-bulk-conseq}
|\Phi(s) - \Phi_0(s)| \leq  \frac{2^{11/4}}{3\pi}s^{3/4}\big(\delta(\delta+2)\big)^{1/4}\big(C_0^2+1\big)^{3/4} < 2s^{3/4}\delta^{1/4}C_1^{3/2}
\end{equation}
and similarly
\begin{equation}\label{eq:Phi-der-conseq}
|\Phi'(s) - \Phi_0'(s)| < s^{-1/4} \delta^{1/4} C_1^{3/2} + 2.1\cdot s^{1/2} \delta^{1/7} C_1^{9/7}.
\end{equation}
We also remark that $\|X\|_3 \geq \|X\|_2 \geq \|\xi\|_2 - \|X-\xi\|_2 = 1 - \delta \geq 1-10^{-38}$.  

We break the argument into several regimes for the parameter $s$.

\noindent \emph{Large $s$.} With hindsight, we set 
\begin{equation}
s_0 = \max\big\{10^6(\E|X|^3)^2, 2\log C_1\big\}
\end{equation}
In particular, $s_0 \geq 10^5$. Using Lemma \ref{lm:Phi-big-s}, that is 
\begin{align*}
\Phi(s) \leq &\sqrt{\frac{6}{\pi}}\Big((1-\delta\sqrt{3})^2-\theta\E|X|^3\Big)^{-1/2} \\
 &+ \sqrt{\frac{6}{\pi}}\exp\left\{-s\left(\frac{\theta^2}{6}-26\delta(\delta+2)\right)\right\}+ 2C_0\left(\sqrt{s}+\frac{2}{\sqrt{s}}\right)e^{-s} = A_1 + A_2 + A_3,
\end{align*}
we will show that $\Phi(s) \leq \Phi(2)$ for all $s \geq s_0$. We take $\theta = \frac{1}{100\E|X|^3}$ which satisfies the conditions of the lemma and then, for the first term $A_1$,  we use
\[
A_1 = \sqrt{\frac{6}{\pi}}\Big((1-\delta\sqrt{3})^2-\theta\E|X|^3\Big)^{-1/2} \leq \sqrt{\frac{6}{\pi}}\big(1-0.01\big)^{-1/2} < \sqrt{2} - \frac{1}{50}.
\]
Thanks to \eqref{eq:Phi-bulk-conseq}, we also have
\[
\sqrt{2} = \Phi_0(2) \leq \Phi(2) + 2^{7/4}\delta^{1/4}C_1^{3/2} = \Phi(2) + A_4,
\]
so it suffices to show that each of the second and third terms $A_2$, $A_3$ as well as this additional error $A_4$ do not exceed $\frac{1}{150}$. Using $\delta < 10^{-38}C_1^{-9}$, we get
\[
A_4 \leq 2^{7/4}\cdot 10^{-19/2}C_1^{-3/4} < \frac{1}{150}.
\]
For the exponent in the second term $A_2$, observe that
\[
26\delta(\delta+2) < 53\delta < 53\cdot 10^{-38}C_1^{-9}(\E|X|^3)^{-6} \leq 10^{-36}(\E|X|^3)^{-2},
\]
and, consequently,
\[
\frac{\theta^2}{6}-26\delta(\delta+2) \geq \frac{1}{6\cdot 10^4(\E|X|^3)^2} -  \frac{1}{10^{36}(\E|X|^3)^2} \geq  \frac{1}{10^5(\E|X|^3)^2}.
\]
Thus, using $s \geq s_0 \geq 10^6(\E|X|^3)^2$, we get
\[
A_2 \leq  \sqrt{\frac{6}{\pi}}\exp\left\{-\frac{s_0}{10^5(\E|X|^3)^2}\right\} \leq  \sqrt{\frac{6}{\pi}}\exp\{-10\} < \frac{1}{150}.
\]
Finally, for the third term, since $s \geq s_0 \geq 10^5$, 
\[
\left(\sqrt{s}+\frac{2}{\sqrt{s}}\right)e^{-s} \leq (\sqrt{s} + 1)e^{-s} \leq e^{\sqrt{s}-s} \leq \frac{1}{300}e^{-s/2},
\]
therefore, since $s \geq s_0 \geq 2\log C_1$,
\[
A_3 \leq 2C_1\left(\sqrt{s}+\frac{2}{\sqrt{s}}\right)e^{-s} \leq \frac{C_1}{150}e^{-s/2} \leq \frac{1}{150}.
\]

\noindent \emph{Moderate $s$.} We now assume that $2.01 \leq s \leq s_0$. Using \eqref{eq:Phi-bulk-conseq} twice and Lemma \ref{lm:Phi0},
\begin{align*}
\Phi(s) \leq \Phi_0(s) + 2s_0^{3/4}\delta^{1/4}C_1^{3/2} &\leq \Phi_0(2) - 2\cdot 10^{-4} + 2s_0^{3/4}\delta^{1/4}C_1^{3/2} \\
&\leq \Phi(2) - 2\cdot 10^{-4} +  2\cdot2^{3/4}\delta^{1/4}C_1^{3/2}  +  2s_0^{3/4}\delta^{1/4}C_1^{3/2} \\
&\leq \Phi(2) - 2\cdot 10^{-4} +  3s_0^{3/4}\delta^{1/4}C_1^{3/2}.
\end{align*}
Inserting the bound on $\delta$,
\begin{align*}
3s_0^{3/4}\delta^{1/4}C_1^{3/2} \leq 3\cdot 10^{-19/2} C_1^{-3/4} s_0^{3/4}\cdot \min\big\{(\E|X|^3)^{-3/2},1\big\} 
\end{align*}
If $s_0 = 10^6(\E|X|^3)^2$, then using the $(\E|X|^3)^{-3/2}$ term in the minimum and $C_1^{-3/4} \leq 1$, we get the above bounded by $3\cdot 10^{-19/2+9/2} = 3\cdot 10^{-5}$. If $s_0 = 2\log C_1$, then using the other term in the minimum, we get the bound by $3\cdot 2^{3/4}10^{-19/2} C_1^{-3/4} (\log C_1)^{3/4} < 3(2/e)^{3/4}10^{-19/2} < 10^{-4}$ since $u^{-1}\log u \leq e^{-1}$ for $u > 1$. In either case,  we get the conclusion $\Phi(s) \leq \Phi(2)$.

\noindent \emph{Small $s$.} We finally assume that $2 \leq s \leq 2.01$. To argue that $\Phi(s) \leq \Phi(2)$, we will show that $\Phi'(s) < 0$. By virtue of \eqref{eq:Phi-der-conseq} and Lemma \ref{lm:Phi0-der},
\begin{align*}
\Phi'(s) &\leq \Phi_0'(s) +  s^{-1/4} \delta^{1/4} C_1^{3/2} + 2.1\cdot s^{1/2} \delta^{1/7} C_1^{9/7} \\
&< -0.02 + (\delta C_1^{6})^{1/4} + 3 \big(\delta C_1^{9}\big)^{1/7}.
\end{align*}
Since $\delta C_1^6 \leq \delta C_1^9 \leq 10^{-38}$, this is clearly negative and the proof is complete.\hfill$\square$


\section{Concluding remarks}

\begin{remark}
Assumption \eqref{eq:assump-S} seems natural: plainly, there are distributions which are \emph{not} close to the Rademacher one, for which the unit vector attaining $\inf \E|\sum a_jX_j|$ is different than $a = (\frac{1}{\sqrt2},\frac{1}{\sqrt2}, 0, \dots, 0)$, for instance it is $a = (1,0,\dots, 0)$ for Gaussian mixtures (see \cite{AH, ENT1}), or for the Rademacher distribution with a large atom at $0$ (see Theorem 4 and Remark 14 in \cite{HT}). 
\end{remark}

\begin{remark}
Handling the complementary case $\|a\|_\infty > \frac{1}{\sqrt2}$ which is not covered by Theorems~\ref{thm:mainS} and~\ref{thm:mainB} is a different story. The trivial convexity argument presented in the introduction works in fact only for the Rademacher case, as it requires $\frac{1}{\sqrt2}\E|X_1| \geq \E\left|\frac{X_1+X_2}{\sqrt{2}}\right|$, and only for the $L_1$-norm (see Remark 21 in \cite{CKT}). To circumvent this, several different approaches have been used: Haagerup's ad hoc approximation (see \S 3 in \cite{Haa}), Nazarov and Podkorytov's induction with a strengthened hypothesis (see Ch.~II, Step 5 in \cite{NP}) which has also been adapted to other distributions (see \cite{CKT, CGT, CST}),  and very recently a different inductive scheme near the extremiser (without a strengthening) needed in a geometric context (see \cite{ENT}).  None of these techniques appears amenable to the broad setting of general distributions that is treated in this paper.
\end{remark}

\begin{remark}
De, Diakonikolas and Servedio obtained in \cite{DDS} a stable version of Szarek's inequality \eqref{eq:Sza} with respect to the unit vector $a$, namely
\begin{equation} \label{eq:dds}
\E\left|\sum_{j=1}^n a_j\e_j\right| \geq \E\left|\frac{\e_1+\e_2}{\sqrt{2}}\right| + \kappa\sqrt{\delta(a)}
\end{equation}
for a universal positive constant $\kappa$,  where the deficit is given by $\delta(a) = |a - (\frac{1}{\sqrt2}, \frac{1}{\sqrt2}, 0, \dots, 0)|^2$,  assuming that $a_1 \geq a_2 \geq \dots \geq a_n \geq 0$. Note that in the setting of Theorem~\ref{thm:mainS},  we have
\[
\left|\E\left|\sum_{j=1}^n a_jX_j\right| - \E\left|\sum_{j=1}^n a_j\e_j\right|\right| \leq \delta_0,
\]
by a simple application of the triangle inequality and $\|\cdot\|_1 \leq \|\cdot\|_2$. Thus, applying this (twice) and the bound \eqref{eq:dds} of De Diakonikolas and Servedio, we conclude that Theorem \ref{thm:mainS} also holds for unit vectors $a$ with $\delta(a) \geq (2\delta_0/\kappa)^2$. The same will apply to Theorem~\ref{thm:mainB} with the aid of Theorem 1.2 from \cite{CNT}, a strengthening of Ball's inequality \eqref{eq:Ball}  (see also \cite{MR}). See \cite{ENT} for numerical values of the constants $\kappa$.
\end{remark}

\begin{remark}
We have used the $\msf{W}_2$-distance in Theorems \ref{thm:mainS} and \ref{thm:mainB} for concreteness and convenience. Of course, for every $p \geq 1$, if we use the $\msf{W}_p$-distance in \eqref{eq:assump-S} and assume that $X_1$ is in $L_{\frac{p}{p-1}}$, then the proofs of Lemmas \ref{lm:phi-unif} and \ref{lm:phi-unif-vec} go through with the Cauchy--Schwarz inequality replaced by H\"older's inequality and the rest of the proof remains unchanged. It might be of interest to examine weaker distances in such statements.
\end{remark}

{\red
\begin{remark}\label{rem:other-q}
Szarek's sharp $L_1 - L_2$ inequality \eqref{eq:Sza} was extended to sharp $L_p - L_2$ bounds for all $p > 0$ by Haagerup in \cite{Haa}, using Fourier-integral representations of $|x|^p$. It therefore seems plausible that our techniques allow to extend Theorem \ref{thm:mainS} to sharp bounds on $L_p$ norms, but additional (nontrivial and technical) work is needed to treat the  analogues of the special function $\Psi_0$, \eqref{eq:defPsi0}, relevant to Haagerup's $L_p$ bounds. Similarly, the main result from \cite{CKT} which extends \eqref{eq:Ball} to sharp $L_p - L_2$ bounds for all $-1 < p < 0$ could be a starting point for extensions of Theorem \ref{thm:mainB} to $L_p$ norms with $-1 < p < 0$.
\end{remark}
}

\subsubsection*{Statements}

The authors state that there is no conflict of interest. This manuscript has no associated data.


\end{document}